\documentclass[11pt]{article}

\usepackage{epsfig}
\usepackage{latexsym}
\usepackage{amssymb,amsfonts,amsthm,amsmath,graphicx,url}
\usepackage{mathrsfs}
\usepackage{color}

\input amssym.def
\input amssym.tex

\def\PP{{\mathcal P}}

\def\Ga{\Gamma}

\def\Si{\Sigma}
\def\Om{\Omega}

\def\a{\alpha}
\def\b{\beta}

\def\Aut{{\rm Aut}}
\def\val{{\rm val}}

\newcommand{\be}{\begin{equation}}
\newcommand{\ee}{\end{equation}}
\newcommand{\bea}{\begin{eqnarray}}
\newcommand{\eea}{\end{eqnarray}}
\newcommand{\bean}{\begin{eqnarray*}}
\newcommand{\eean}{\end{eqnarray*}}

\definecolor{VeryLightBlue}{rgb}{0.9,0.9,1}
\definecolor{LightBlue}{rgb}{0.8,0.8,1}
\definecolor{MidBlue}{rgb}{0.5,0.5,1}
\definecolor{DarkBlue}{rgb}{0,0,0.6}
\definecolor{Blue}{rgb}{0,0,1}

\definecolor{Gold}{rgb}{1,0.843,0}

\definecolor{LightGreen}{rgb}{0.88,1,0.88}
\definecolor{MidGreen}{rgb}{0.6,1,0.6}
\definecolor{DarkGreen}{rgb}{0,0.6,0}

\definecolor{VeryLightYellow}{rgb}{1,1,0.9}
\definecolor{LightYellow}{rgb}{1,1,0.6}
\definecolor{MidYellow}{rgb}{1,1,0.5}
\definecolor{DarkYellow}{rgb}{1,1,0.2}

\definecolor{DarkPurple}{rgb}{.6,0,1}

\definecolor{Red}{rgb}{1,0,0}
\definecolor{VeryLightRed}{rgb}{1,0.9,0.9}
\definecolor{LightRed}{rgb}{1,0.8,0.8}
\definecolor{MidRed}{rgb}{1,0.55,0.55}

\long\def\delete#1{}

\newtheorem{thm}{Theorem}[section]

\newtheorem{conj}[thm]{Conjecture}
\newtheorem{lem}[thm]{Lemma}

\begin{document}

\title{Nowhere-zero 3-flows in graphs admitting solvable arc-transitive groups of automorphisms
\thanks{Li was supported by the National Natural Science Foundation of China (11171129). Zhou was supported by a Future Fellowship of the Australian Research Council. Part of the work was done when Zhou visited Huazhong Normal University in 2011. The authors are grateful to an anonymous referee for his/her comments which led to improved presentation.}}
\author{\textbf{Xiangwen Li} \\
{\em Department of Mathematics, Huazhong Normal University}\\
{\em Wuhan 430079, China} \\ 
{\em E-mail: xwli68@mail.ccnu.edu.cn} \\ \\
\textbf{Sanming Zhou} \\ 
{\em Department of Mathematics and Statistics, The University of Melbourne} \\ 
{\em Parkville, VIC3010, Australia} \\
{\em E-mail: smzhou@ms.unimelb.edu.au}}

\date{}

\openup 1.0\jot

\maketitle

\vspace{-0.8cm}

\begin{abstract}
Tutte's 3-flow conjecture asserts that every 4-edge-connected graph has a nowhere-zero 3-flow. In this note we prove that every regular graph of valency at least four admitting a solvable arc-transitive group of automorphisms admits a nowhere-zero 3-flow.  

\medskip
\emph{Key words:} integer flow; nowhere-zero 3-flow; vertex-transitive graph; arc-transitive graph; solvable group 

\medskip
\emph{AMS Subject Classification (2010):} 05C21, 05C25
\end{abstract}

\section{Introduction}
\label{sec:int}

All graphs in this paper are finite and undirected, and all groups considered are finite. 
Let $\Gamma = (V(\Ga), E(\Ga))$ be a graph endowed with an orientation. For an integer $k \ge 2$, a \emph{$k$-flow} \cite{BM76} in $\Gamma$ is an integer-valued function $f: E(\Gamma)\rightarrow\{0,\pm 1,\pm2,\ldots,\pm (k-1)\}$ such that, for every $v\in V(\Gamma)$, 
$$
\sum_{e\in E^+(v)}f(e)=\sum_{e\in E^-(v)}f(e),
$$
where $E^+(v)$ is the set of edges of $\Ga$ with tail $v$ and $E^-(v)$ the set of edges of $\Ga$ with head $v$. 
A $k$-flow $f$ in $\Ga$ is called a \emph{nowhere-zero $k$-flow} if $f(e)\neq 0$ for every $e \in E(\Gamma)$. Obviously, if $\Gamma$ admits a nowhere-zero $k$-flow, then $\Gamma$ admits a nowhere-zero $(k+1)$-flow. It is also easy to see that whether a graph admits a nowhere-zero $k$-flow is independent of its orientation. The notion of nowhere-zero flows was introduced by Tutte in \cite{Tu54,Tu66} who proved that a planar graph admits a nowhere-zero 4-flow if and only if the Four Color Conjecture holds. The reader is referred to Jaeger \cite{Jaeger88} and Zhang \cite{Zhang97} for surveys on nowhere-zero flows and to \cite[Chapter 21]{BM76} for an introduction to this area. 

In \cite{Tu54,Tu66} Tutte proposed three celebrated conjectures on integer flows which are still open in general. One of them is the following well-known 3-flow conjecture (see e.g. \cite[Conjecture 21.16]{BM76}).

\begin{conj}
\label{conj-tutte}
(Tutte's 3-flow conjecture)
Every 4-edge-connected graph admits a nowhere-zero 3-flow.
\end{conj}

This conjecture has been extensively studied in over four decades; see e.g. \cite{Fan081, Fan082, Lai3, Li12, Wu, Thom12, Yan11, YL11}. Solving a long-standing conjecture by Jaeger \cite{Jaeger79} (namely the weak 3-flow conjecture), recently Thomassen \cite{Thom12} proved that every 8-edge-connected graph admits a nowhere-zero 3-flow. This breakthrough was further improved by Lov\'{a}sz, Thomassen, Wu and Zhang \cite{Wu} who proved the following result.

\begin{thm}
\label{thm:LTWZ}
(\cite[Theorem 1.7]{Wu}) Every 6-edge-connected graph admits a nowhere-zero 3-flow.
\end{thm}

It is well known \cite{Watkins} that every vertex-transitive graph of valency $d \ge 1$ is $d$-edge-connected. Thus, when restricted to the class of vertex-transitive graphs, Conjecture \ref{conj-tutte} asserts that every vertex-transitive graph of valency at least four admits a nowhere-zero 3-flow. Due to Theorem \ref{thm:LTWZ} this is now boiled down to vertex-transitive graphs of valency $5$, since every regular graph with even valency admits a nowhere-zero 2-flow. In an attempt to Tutte's 3-flow conjecture for Cayley graphs, Poto\v{c}nik, \v{S}koviera and \v{S}kerkovski \cite{Potocnik05} proved the following result. (It is well known \cite{Biggs} that every Cayley graph is vertex-transitive, but the converse is not true.)

\begin{thm}
\label{thm:cay-abelian}
(\cite[Theorem 1.1]{Potocnik05})
Every Cayley graph of valency at least four on an abelian group admits a nowhere-zero 3-flow.
\end{thm}

This was generalized by N\'{a}n\'{a}siov\'{a} and \v{S}koviera \cite{Nanasiova} to Cayley graphs on nilpotent groups. 

\begin{thm}
\label{thm:cay-nil}
(\cite[Theorem 4.3]{Nanasiova})
Every Cayley graph of valency at least four on a nilpotent group admits a nowhere-zero 3-flow.
\end{thm}

It would be nice if one can prove Tutte's 3-flow conjecture for all vertex-transitive graphs. As a step towards this, we prove the following result in the present paper. 

\begin{thm}
\label{thm:vt-3}
Every regular graph of valency at least four admitting a solvable arc-transitive group of automorphisms admits a nowhere-zero 3-flow. 
\end{thm}

Note that any $G$-arc-transitive graph is necessarily $G$-vertex-transitive and $G$-edge-transitive. On the other hand, any $G$-vertex-transitive and $G$-edge-transitive graph with odd valency is $G$-arc-transitive. (This result is due to Tutte, and its combinatorial proof given in \cite[Proposition 1.2]{Cheng-Oxley} can be easily extended from $G = \Aut(\Ga)$ to a subgroup $G$ of $\Aut(\Ga)$.) Therefore, Theorem \ref{thm:vt-3} is equivalent to the following: For any solvable group $G$, every $G$-vertex-transitive and $G$-edge-transitive graph with valency at least four admits a nowhere-zero 3-flow. 

In Section \ref{sec:proof} we will prove a weaker version (Claim 1) of Theorem \ref{thm:vt-3}, which together with Theorem \ref{thm:LTWZ} implies Theorem \ref{thm:vt-3}. Note that none of Theorems \ref{thm:vt-3} and \ref{thm:cay-nil} is implied by the other, because not every Cayley graph is arc-transitive, and on the other hand an arc-transitive graph is not necessarily a Cayley graph (see e.g. \cite{Biggs, Praeger97}).

At this point we would like to mention a few related results. Alspach and Zhang (1992) conjectured that every Cayley graph with valency at least two admits a nowhere-zero 4-flow. Since every 4-edge-connected graph admits a nowhere-zero 4-flow \cite{Jaeger79}, this conjecture is reduced to the cubic case. Alspach, Liu and Zhang \cite[Theorem 2.2]{ALZ} confirmed this conjecture for cubic Cayley graphs on solvable groups. This was then improved by Nedela and \v{S}koviera \cite{NS} who proved that any counterexample to the conjecture of Alspach and Zhang must be a regular cover over a Cayley graph on an almost simple group. (A group $G$ is {\em almost simple} if it satisfies $T \le G \le \Aut(T)$ for some simple group $T$.) In \cite{P}, Pota\v{c}nik proved that every connected cubic graph that admits a solvable vertex-transitive group of automorphisms is $3$-edge-colourable or isomorphic to the Petersen graph. This is another generalization of the result of Alspach, Liu and Zhang above, because Petersen graph is not a Cayley graph, and for cubic graphs $3$-edge-colourability is equivalent to the existence of a nowhere-zero 4-flow.

It would be pleasing if one can replace arc-transitivity by vertex-transitivity in Theorem \ref{thm:vt-3}. As an intermediate step towards this, one may try to prove that every Cayley graph of valency at least four on a solvable group admits a nowhere-zero 3-flow, thus generalizing Theorem \ref{thm:cay-nil} and the above-mentioned result of Alspach, Liu and Zhang simultaneously.

\section{Preparations}
\label{sec:prep}

We follow \cite{BM76} and \cite{Dixon-Mortimer, Rotman} respectively for graph- and group-theoretic terminology and notation. The derived subgroup of a group $G$ is defined as $G' := [G, G]$, the subgroup of $G$ generated by all commutators $x^{-1}y^{-1}xy$, $x, y \in G$. Define $G^{(0)} := G, G^{(1)} := G'$ and $G^{(i)} := (G^{(i-1)})'$ for $i \ge 1$. A group $G$ is {\em solvable} if there exists an integer $n \ge 0$ such that $G^{(n)} = 1$; in this case the least integer $n$ with $G^{(n)} = 1$ is called the {\em derived length} of $G$. \delete{Equivalently, a solvable group can be defined as a group $G$ that has an abelian series, that is, a normal series $G = G_0 \unrhd G_1 \unrhd \cdots \unrhd G_r = 1$ such that $G_{i-1}/G_i$ is abelian for $i = 1, \ldots, r$.} Solvable groups with derived length 1 are precisely nontrivial abelian groups. In the proof of Theorem \ref{thm:vt-3} we will use the fact that any solvable group contains an abelian normal subgroup with respect to which the quotient group has a smaller derived length. 

All definitions in the next three paragraphs are standard and can be found in \cite[Part Three]{Biggs} or \cite{Praeger97}. 

Let $G$ be a group acting on a set $\Om$. That is, for each $(\a, g) \in \Om \times G$, there corresponds an element $\a^g \in \Om$ such that $\a^1 = \a$ and $(\a^{g})^h = \a^{gh}$ for any $\a \in \Om$ and $g, h \in G$, where $1$ is the identity element of $G$. We say that $G$ is {\em transitive} on $\Om$ if for any $\a, \b \in \Om$ there exists at least one element $g \in G$ such that $\a^g = \b$, and {\em regular} if for any $\a, \b \in \Om$ there exists exactly one element $g \in G$ such that $\a^g = \b$. The group $G$ is {\em intransitive} on $\Om$ if it is not transitive on $\Om$. A partition $\PP$ of $\Om$ is {\em $G$-invariant} if $P^g := \{\a^g: g \in G\} \in \PP$ for any $P \in \PP$ and $g \in G$, and {\em nontrivial} if $1 < |P| < |\Om|$ for every $P \in \PP$. 

Suppose that $\Ga$ is a graph admitting $G$ as a group of automophisms. That is, $G$ acts on $V(\Ga)$ (not necessarily faithfully) such that, for any $\a, \b \in V(\Ga)$ and $g \in G$, $\a$ and $\b$ are adjacent in $\Ga$ if and only if $\a^g$ and $\b^g$ are adjacent in $\Ga$. (If $K$ is the {\em kernel} of the action of $G$ on $V(\Ga)$, namely, the subgroup of all elements of $G$ that fix every vertex of $\Ga$, then $G/K$ is isomorphic to a subgroup of the automorphism group $\Aut(\Ga)$ of $\Ga$.) We say that $\Gamma$ is \emph{$G$-vertex-transitive} if $G$ is transitive on $V(\Gamma)$, and {\em $G$-edge-transitive} if $G$ is transitive on the set of edges of $\Ga$. If $\Ga$ is $G$-vertex-transitive such that $G$ is also transitive on the set of arcs of $\Ga$, then $\Ga$ is called {\em $G$-arc-transitive}, where an {\em arc} is an ordered pair of adjacent vertices.  

Let $\Ga$ be a graph and $\PP$ a partition of $V(\Ga)$. The \emph{quotient graph} of $\Ga$ with respect to $\PP$, denoted by $\Ga_{\PP}$, is the graph with vertex set $\PP$ in which $P, Q \in \PP$ are adjacent if and only if there exists at least one edge of $\Ga$ joining a vertex of $P$ and a vertex of $Q$. For blocks $P, Q \in \PP$ adjacent in $\Ga_{\PP}$, denote by $\Ga[P, Q]$ the bipartite subgraph of $\Ga$ with vertex set $P \cup Q$ whose edges are those of $\Ga$ between $P$ and $Q$. In the case when all blocks of $\PP$ are independent sets of $\Ga$ and $\Ga[P, Q]$ is a $t$-regular bipartite graph for each pair of adjacent $P, Q \in \PP$, where $t \ge 1$ is an integer independent of $(P, Q)$, we say that $\Ga$ is a \emph{multicover} of $\Ga_{\PP}$. A multicover with $t = 1$ is thus a topological cover in the usual sense. In the proof of Theorem \ref{thm:vt-3}, we will use the following lemma in the case when $k = 3$. 
 
\begin{lem}
\label{lem:multi}
Let $k \ge 2$ be an integer. If a graph admits a nowhere-zero $k$-flow, then its multicovers all admit a nowhere-zero $k$-flow.
\end{lem}  
\begin{proof}
Using the notation above, let $\Ga$ be a multicover of $\Si := \Ga_{\PP}$. Suppose that $\Si$ admits a nowhere-zero $k$-flow $f$ (with respect to some orientation). For each oriented edge $(P, Q)$ of $\Si$, orient the edges of the $t$-regular bipartite graph $\Ga[P, Q]$ in such a way that they all have tails in $P$ and heads in $Q$, and then assign $f(P, Q)$ to each of them. Denote this nowhere-zero function on the oriented edges of $\Ga$ by $g$, and denote the oriented edge of $\Ga$ with tail $\a$ and head $\b$ by $(\a, \b)$. It can be verified that, for any $P \in \PP$ and $\a \in P$, $\sum_{(\a, \b) \in E^+_{\Ga}(\a)} g(\a, \b) = \sum_{(P, Q) \in E^+_{\Si}(P)} t \cdot f(P, Q)$ and $\sum_{(\b, \a) \in E^-_{\Ga}(\a)} g(\b, \a) = \sum_{(Q, P) \in E^-_{\Si}(P)} t \cdot f(Q, P)$. Since $f$ is a nowhere-zero $k$-flow in $\Si$, for every $P \in \PP$, we have 
$$\sum_{(P, Q) \in E^+_{\Si}(P)} f(P, Q) = \sum_{(Q, P) \in E^-_{\Si}(P)} f(Q, P).$$  Therefore, for every $\a \in V(\Ga)$, we have $$\sum_{(\a, \b) \in E^+_{\Ga}(\a)} g(\a, \b) = \sum_{(\b, \a) \in E^-_{\Ga}(\a)} g(\b, \a)$$ and so $g$ is a nowhere-zero $k$-flow in $\Ga$. 
\end{proof}

If $\Ga$ is a $G$-vertex-transitive graph, then for any normal subgroup $N$ of $G$, the set $\PP_N := \{\a^N: \a \in V(\Ga)\}$ of $N$-orbits on $V(\Ga)$ is a $G$-invariant partition of $V(\Ga)$, called a {\em $G$-normal partition} of $V(\Ga)$ \cite{Praeger97}, where $\a^N := \{\a^g: g \in N\}$. Denote the corresponding quotient graph by $\Ga_N := \Ga_{\PP_N}$. The quotient group $G/N$ induces an action on $\PP_N$ defined by $(\a^N)^{Ng} = (\a^g)^{N}$. The following observations can be easily proved (see e.g. \cite{Praeger97}).  
 
\begin{lem}
\label{lem:vt-bi}
(\cite{Praeger97}) 
Let $\Ga$ be a connected $G$-vertex-transitive graph, and $N$ a normal subgroup of $G$ that is intransitive on $V(\Ga)$. Then the following hold:
\begin{itemize}
\item[\rm (a)] $\Ga_{N}$ is $G/N$-vertex-transitive under the induced action of $G/N$ on $\PP_N$;
\item[\rm (b)]  for $P, Q \in \PP_N$ adjacent in $\Ga_{N}$, $\Ga[P, Q]$ is a regular subgraph of $\Ga$; 
\item[\rm (c)]  if in addition $\Ga$ is $G$-arc-transitive, then $\Ga_{N}$ is $G/N$-arc-transitive and $\Ga$ is a multicover of $\Ga_{N}$.
\end{itemize} 
\end{lem}

\section{Proof of Theorem \ref{thm:vt-3}}
\label{sec:proof}

By Theorem \ref{thm:LTWZ}, in order to prove Theorem \ref{thm:vt-3} it suffices to prove that, for any finite solvable group $G$, every $G$-arc-transitive graph of valency five admits a nowhere-zero 3-flow. We will prove the following seemingly stronger but equivalent result: 

\medskip
\textbf{Claim 1.}~ For any solvable group $G$, every $G$-arc-transitive graph with valency at least four and not divisible by three admits a nowhere-zero 3-flow. 
\medskip

We will prove this by induction on the derived length of the solvable group $G$, using Theorem \ref{thm:cay-abelian} as the base case. We will use the following fact \cite[Lemma 16.3]{Biggs}: a graph is isomorphic to a Cayley graph if and only if its automorphism group contains a subgroup that is regular on the vertex set. Denote by $\val(\Ga)$ the valency of a regular graph $\Ga$. 

Without loss of generality we may assume that the solvable group $G$ is faithful on the vertex set of the graph under consideration for otherwise we can replace $G$ by its quotient group (which is also solvable) by the kernel of $G$ on the vertex set. Under this assumption $G$ is isomorphic to a subgroup of the automorphism group of the graph. We may also assume that the graph under consideration is connected (for otherwise we consider its components). We make induction on the derived length $n(G)$ of $G$. 

Suppose that $n(G) = 1$ and $\Ga$ is a $G$-arc-transitive graph with $\val(\Ga) \ge 4$. Then $G$ is abelian and so is regular on $V(\Ga)$. (A transitive abelian group must be regular.) Since $G$ is isomorphic to a subgroup of $\Aut(\Ga)$, it follows that $\Ga$ is isomorphic to a Cayley graph on $G$. Thus, by Theorem \ref{thm:cay-abelian}, $\Ga$ admits a nowhere-zero 3-flow.

Assume that, for some integer $n \ge 1$, the result (in Claim 1) holds for any solvable group of derived length at most $n$. Let $G$ be a solvable group with derived length $n(G) = n+1$. Let $\Ga$ be a connected $G$-arc-transitive graph such that $\val(\Ga) \ge 4$ and $\val(\Ga)$ is not divisible by $3$. If $\val(\Ga)$ is even, then $\Ga$ admits a nowhere-zero 2-flow and hence a nowhere-zero 3-flow. So we assume that $\val(\Ga) \ge 5$ is odd. Since $3$ does not divide $\val(\Ga)$ by our assumption, every prime factor of $\val(\Ga)$ is no less than $5$. Since $G$ is solvable, it contains an abelian normal subgroup $N$ such that the quotient group $G/N$ has derived length at most $n(G) - 1 = n$. Note that $G/N$ is solvable (as any quotient group of a solvable group is solvable) and $N \ne 1$ (for otherwise $G/N \cong G$ would have derived length $n(G)$). If $N$ is transitive on $V(\Ga)$, then it is regular on $V(\Ga)$ as $N$ is abelian. In this case $\Ga$ is isomorphic to a Cayley graph on $N$ and so admits a nowhere-zero 3-flow by Theorem \ref{thm:cay-abelian}. 

In what follows we assume that $N$ is intransitive on $V(\Ga)$. By Lemma \ref{lem:vt-bi}, $\Ga_{N}$ is a connected $G/N$-arc-transitive graph, and $\Ga$ is a multicover of $\Ga_{N}$. Thus $\val(\Ga_{N})$ is a  divisor of $\val(\Ga)$ and so is not divisible by $3$. If $\val(\Ga_{N}) = 1$, then $\Ga$ is a regular bipartite graph of valency at least two and so admits a nowhere-zero 3-flow \cite{BM76}. Assume that $\val(\Ga_{N}) > 1$. Then $\val(\Ga_{N}) \ge 5$ and every prime factor of $\val(\Ga_{N})$ is no less than $5$. Thus, since $G/N$ is solvable of derived length at most $n$, by the induction hypothesis, $\Ga_{N}$ admits a nowhere-zero 3-flow. Since $\Ga$ is a multicover of $\Ga_{N}$, by Lemma \ref{lem:multi}, $\Ga$ admits a nowhere-zero 3-flow. This completes the proof of Claim 1 and hence the proof of Theorem \ref{thm:vt-3}.

\end{document}